\documentclass[11pt,twoside]{amsart}
\usepackage{mathrsfs}
\usepackage{amsmath}
\usepackage{amsthm}

\usepackage{amsfonts}
\usepackage{amssymb}
\usepackage{latexsym}
\usepackage[all]{xy}

\date{\empty}
\thanks{}
\allowdisplaybreaks[4] \footskip=15pt
\renewcommand{\uppercasenonmath}[1]{}

\numberwithin{equation}{section} \theoremstyle{plain}
\newtheorem*{thm*}{Main Theorem}
\newtheorem{theorem}{Theorem}[section]
\newtheorem{corollary}[theorem]{Corollary}
\newtheorem*{corollary*}{Corollary}

\newtheorem*{claim*}{Claim}
\newtheorem{lemma}[theorem]{Lemma}
\newtheorem*{lemma*}{Lemma}

\newtheorem*{proposition*}{Proposition}

\newtheorem*{remark*}{Remark}

\newtheorem*{example*}{Example}

\newtheorem*{question*}{Question}
\newtheorem{definition}[theorem]{Definition}
\newtheorem*{definition*}{Definition}

\newtheorem*{acknowledgements*}{ACKNOWLEDGEMENTS}

\begin{document}

\begin{center}
{\large  \bf Representations of the Drazin inverse involving idempotents in a ring}\\
\vspace{0.8cm} {\small \bf Huihui Zhu, Jianlong Chen\footnote{Corresponding author.
Department of Mathematics, Southeast University, Nanjing 210096, China.
Email: jlchen@seu.edu.cn}}
\end{center}

\bigskip
{ \bf  Abstract:}  \leftskip0truemm\rightskip0truemm We present some formulae for the Drazin inverse of difference and product of idempotents in a ring. A number of results of bounded linear operators in Banach spaces are extended to the ring case.
 \bigskip

\section { \bf Introduction}
Let $R$ be an associative ring with unity $1\neq 0$. The symbols $R^{-1}$, $R^D$ and $R^{\rm nil}$ denote the sets of invertible, Drazin invertible and nilpotent elements of $R$, respectively. The commutant of an element $a\in R$ is defined as ${\rm comm}(a)=\{x\in R:xa=ax \}$.

An element $a\in R$ is said to have a  Drazin inverse [6] if there exists $b\in R$ such that
\begin{center}
$b\in {\rm comm}(a)$, $bab=b$, $a-a^2b\in R^{\rm nil}$.
\end{center}
The element $b\in R$ above is unique and denoted by $a^D$. The nilpotency index of $a-a^2b$ is called the Drazin index of $a$, denoted by ${\rm ind}(a)$. If ${\rm ind}(a)=1$, then $a$ is group invertible and the group inverse of $a$ is denoted by $a^\#$. By $a^\pi=1-aa^D$ we mean the spectral idempotent of $a$. It is well known that $a\in R^D$ implies that $a^2\in R^D$ and $(a^2)^D=(a^D)^2$.

Gro{\ss} and Trenkler [7] considered the nonsingularity of $p-q$ for general matrix projectors $p$ and $q$. Koliha and Rako\v{c}evi\'{c} [10] studied the invertibility of the sum $f+g$ when $f$ and $g$ are idempotents in a ring or bounded linear operators in Hilbert or Banach spaces, and proved that $f+g$ is invertible if and only if $f-g$ is invertible. Koliha and Rako\v{c}evi\'{c} [11] obtained the equivalent conditions for the invertibility of $f-g$ in a ring. They also gave applications to bounded linear operators in Banach and Hilbert spaces. Koliha, Rako\v{c}evi\'{c} and Stra\v{s}kraba [12] presented new results on the invertibility of the sum of projectors and obtained the formulae of invertibility of $p-q$ and $p+q$. The problems of Drazin inverse of difference and product of idempotents were studied by many researchers, such as[3,4,5,9,13].

Deng and  Wei [5] presented the formulae for the Drazin inverse of difference and product of idempotent bounded linear operators in Banach spaces. In this paper, we give a algebraic proof of the Drazin inverse of sum, difference and product of idempotents in a ring. Moreover, the formulae of the Drazin inverse involving idempotents are established. Hence, we extend the results in [4,5] to the ring case.

\section{\bf Some lemmas }
In what follows, $p$, $q$ always mean any two idempotents in a ring $R$. We state several known results in the form of lemmas without proofs.

\begin{lemma} $[2, {\rm Proposition}~3.1, {\rm Theorem}~3.3]$ Let $\sum=\{ p-q, 1-pq, p-pq, p-qp, p-pqp, 1-qp, q-pq, q-qp, p+q-pq\}$. If one of the elements in the $\sum$ is Drazin invertible, then all other elements in $\sum$ are Drazin invertible.
\end{lemma}

\begin{lemma}$[2, {\rm Theorem}~3.4]$ The following statements are equivalent:\\
$(1)$ $pq\in R^D,$\\
$(2)$ $1-p-q\in R^D$,\\
$(3)$ $(1-p)(1-q)\in R^D$.
\end{lemma}
\begin{lemma} Let $a, b\in R^D$. Then $(ba)^D=b((ab)^D)^2a$. If $ab=ba$, then $(ab)^D=b^Da^D=a^Db^D$.
\end{lemma}
\begin{lemma}$[1,{\rm Theorem~3.6}]$ Let $a, b\in R$. If $1-ab\in R^D$ with ${\rm ind}(1-ab)=k$, then $1-ba\in R^D$ with ${\rm ind}(1-ba)=k$ and
\begin{center}
$(1-ba)^D=1+b((1-ab)^D-(1-ab)^\pi r)a$,
\end{center}
where $r=\displaystyle{\sum_{i=0}^{k-1}(1-ab)^i}$.
\end{lemma}

\section{\bf Main results}

In this section, we present some formulae on the Drazin inverse of difference and product of idempotents of ring $R$.

\begin{definition} Let $p-q\in R^D$. Define $F$, $G$ and $H$ as
\begin{center}
$F=p(p-q)^D$, $G=(p-q)^D p$, $H=(p-q)^D(p-q)$.
\end{center}
\end{definition}

\begin{theorem} Let $p-q\in R^D$. Then $F$, $G$ and $H$ above are idempotents and\\
$(1)$ $F=(p-q)^D (1-q),$\\
$(2)$ $G=(1-q)(p-q)^D$.
\end{theorem}

\begin{proof} Since $p$, $q$ are idempotents, we get
\begin{center}
$p(p-q)^2=(p-q)^2p=p-pqp$.
\end{center}
Note that $a\in R^D$ and $ab=ba$ imply $a^Db=ba^D$ by [6, Corollary 2]. It follows that $p\in {\rm comm}((p-q)^D)^2$. Hence, we have
\begin{eqnarray*}
F &=& p(p-q)^D=p((p-q)^D)^2(p-q)\\
&=&((p-q)^D)^2p(p-q)=((p-q)^D)^2(p-q)(1-q)\\
 &=&(p-q)^D (1-q).
\end{eqnarray*}

Next, we prove that $F$ is idempotent. From
\begin{center}
$p(p-q)^D=(p-q)^D (1-q)$,
\end{center}
we have
\begin{eqnarray*}
        F^2 &=& (p-q)^D(1-q)p(p-q)^D = (p-q)^D(1-q)(p-q)(p-q)^D \\
         &=& p(p-q)^D(p-q)(p-q)^D = p(p-q)^D\\
         &=&F.
      \end{eqnarray*}

Similarly, $G^2=G=(1-q)(p-q)^D$. It is obvious that $H$ is idempotent and $H=(p-q)(p-q)^D=(p-q)^D(p-q)$.
\end{proof}

We replace $p$ by $q$ in Theorem 3.2 to obtain more relations among $F$, $G$ and $H$.

\begin{corollary} Let $p-q\in R^D$. Then \\
$(1)$ $q(p-q)^D=(p-q)^D(1-p),$\\
$(2)$ $(p-q)^Dq=(1-p)(p-q)^D,$\\
$(3)$ $qH=Hq,$\\
$(4)$ $G(1-q)=(1-q)F.$
\end{corollary}
\begin{proof} (1) We can get (1) and (2) in a similar way of Theorem 3.2.

(3) Since $H=(p-q)^D(p-q)$, we have
\begin{eqnarray*}
  qH &=& q(p-q)^D(p-q)=(p-q)^D(1-p)(p-q) \\
   &=& (p-q)^D(p-q)q\\
   &=&Hq.
\end{eqnarray*}

(4) By Theorem 3.2, we have
\begin{eqnarray*}
  G(1-q) &=& (p-q)^Dp(p-q)=(1-q)(p-q)^D(p-q) \\
   &=& (1-q)(1-q-1+p)(p-q)^D=(1-q)p(p-q)^D\\
   &=&(1-q)F.
\end{eqnarray*}

The proof is complete.
\end{proof}

\begin{theorem} Let $p-q\in R^D$. Then \\
$(1)$ $Fp=pG=pH=Hp,$\\
$(2)$ $qHq=qH=Hq=HqH.$
\end{theorem}

\begin{proof} (1) It is obvious $Fp=pG$, we only need to show $pG=pH$ and $pH=Hp$.
\begin{eqnarray*}
  pG &=& p(p-q)^D p = (p-q)^D(1-q)p \\
   &=& (p-q)^D (p-q)p\\
   &=& Hp.
\end{eqnarray*}

According to Theorem 3.2, we get
\begin{eqnarray*}
  pH &=& p(p-q)^D (p-q)=(p-q)^D(1-q)(p-q) \\
     &=& (p-q)^D (p-q)p\\
     &=&  Hp.
\end{eqnarray*}

Hence, (1) holds.

(2) Note that $qH=Hq$ in Corollary 3.3(3). We obtain that $qHq=(Hq)q=Hq$. Since $H$ is idempotent, $HqH=H^2q=Hq$.
Thus, $qHq=qH=Hq=HqH$.
\end{proof}

The following theorems, the main result of this paper, give the formulae of the Drazin inverses of product and difference of idempotents in a ring $R$.
\begin{theorem} Let $p-q\in R^D$. Then \\
$(1)$ $ (1-pqp)^D=[(p-q)^D]^2p+1-p,$\\
$(2)$ $ (p-pqp)^D= [(p-q)^D]^2p=p[(p-q)^D]^2,$\\
$(3)$ $ (p-pq)^D=p[(p-q)^D]^3,$\\
$(4)$ $ (p-qp)^D=[(p-q)^D]^3p,$\\
$(5)$ If ${\rm ind}(p-q)=k$, then
$$(1-pq)^D=1-p+[(p-q)^D]^2[p+pq(1-p)]+[\sum_{i=0}^{k-1}(p-q)^\pi(p-q)^{2i}]pq(p-1).$$
\end{theorem}

\begin{proof}
(1) Note that $1-pqp=(p-q)^2p+1-p$ and $((p-q)^2)^D=((p-q)^D)^2$. Since $(p-q)^2p(1-p)=(1-p)(p-q)^2p=0$, $(1-pqp)^D= [(p-q)^D]^2p+1-p$ by [6, Corollary 1].

(2) Observing that $p-pqp=p(p-q)^2=(p-q)^2p$, we get $(p-pqp)^D=[(p-q)^D]^2p=p[(p-q)^D]^2$ from Lemma 2.3.

(3) Let $x=p[(p-q)^D]^3$. We prove that $x$ is the Drazin inverse of $p-pq$ by showing the following conditions hold.

(a) From $p(p-q)^2=(p-q)^2p=(p-pq)p$, it follows that
\begin{eqnarray*}
 ~~~ (p-pq)x &=& (p-pq)p[(p-q)^D]^3 = p(p-q)^2[(p-q)^D]^3 \\
  &=& p(p-q)^D
\end{eqnarray*}
and
\begin{eqnarray*}
  x(p-pq) &=& p[(p-q)^D]^3(p-pq)= p(p-q)^D[(p-q)^D ]^2p(p-q)\\
   &=& p(p-q)^D p(p-q)^D= p(p-q)^D\\
   &=&(p-pq)x.
\end{eqnarray*}

(b) Note that $(p-pq)x=p(p-q)^D$. We have
\begin{eqnarray*}
x(p-pq)x &=& p[(p-q)^D]^3p(p-q)^D= [(p-q)^D]^2p(p-q)^Dp(p-q)^D \\
 &=& [(p-q)^D]^2p(p-q)^D=p[(p-q)^D]^3 \\
 &=& x.
\end{eqnarray*}

(c) Since $x(p-pq)^D=p(p-q)^D$, we obtain that
\begin{eqnarray*}
  (p-pq)-(p-pq)^2x &=& (p-pq)-(p-pq)p(p-q)^D \\
   &=& p(p-q)-p(p-q)^2(p-q)^D \\
   &=& p(p-q)(p-q)^\pi.
\end{eqnarray*}

 According to $pH=Hp$ and $qH=Hq$, it follows that $p(p-q)(p-q)^\pi=(p-q)^\pi p(p-q)$. By induction, $(p(p-q))^m=p(p-q)^{2m-1}$. Take $m \geqslant {\rm ind}(p-q)$, then
\begin{center}
$[(p(p-q)(p-q)^\pi)]^m=p(p-q)^{2m-1}(p-q)^\pi=0$.
\end{center}

Hence, $(p-pq)-(p-pq)^2x$ is nilpotent.

Therefore, $(p-pq)^D=p[(p-q)^D]^3$.

(4) Use a similar proof of (3).

(5) Since $p-q\in R^D$, $1-pq \in R^D$ according to Lemma 2.1. By Lemma 2.4, we obtain
\begin{center}
$(1-pq)^D=1+p[(1-pqp)^D-(1-pqp)^\pi r]pq$,
\end{center}
 where $r=\displaystyle{\sum_{i=0}^{k-1}}(1-pqp)^i$. Note that (1). We have
$$(1-pq)^D=1-p+[(p-q)^D]^2[p+pq(1-p)]+[\sum_{i=0}^{k-1}(p-q)^\pi(p-q)^{2i}]pq(p-1).$$
\end{proof}

\begin{theorem} Let $1-p-q\in R^D$. Then\\
$(1)$ $(pqp)^D=[(1-p-q)^D]^2p=p[(1-p-q)^D]^2,$\\
$(2)$ $(pq)^D=[(1-p-q)^D]^4pq.$
\end{theorem}

\begin{proof} (1) By $pqp=p(1-p-q)^2=(1-p-q)^2p$ and Lemma 2.2, it follows that $(pqp)^D=[(1-p-q)^D]^2p=p[(1-p-q)^D]^2$.

(2) From $pq=ppq$ and Lemma 2.3, we have
\begin{center}
$(pq)^D=p[(pqp)^D]^2pq=[(pqp)^D]^2pq$.
\end{center}

According to (1), we obtain
 \begin{center}
 $(pq)^D=[(pqp)^D]^2pq=[(1-p-q)^D]^4pq$.
 \end{center}
 \end{proof}
\begin{theorem} Let $pq\in R^D$. Then\\
$(1)$ $(pq)^D=qp$ if and only if $pq=qp,$\\
$(2)$ $(pq)^D=(pqp)^D-p((1-q)(1-p))^D,$\\
$(3)$ $(pq)^Dpq=(pqp)^Dpq.$
\end{theorem}
\begin{proof}

(1) If $pq=qp$, we can get $(pq)^D=qp$ by a direct calculation.

Conversely, since $(pq)^D=qp$, we obtain $pqp=qpq$ and $qp=qpqp$. On the other hand, we also get $pqpq=pq$. Hence,
\begin{center}
$pq=pqpq=ppqp=pqp=qpq=qpqp=qp$.
\end{center}

(2) By Theorem 3.5(4), we have  $(p-qp)^D=((p-q)^D)^3p$ and $$(q-pq)^D=((q-p)^D)^3q=-((p-q)^D)^3q.$$

Hence,
$$(q-pq)^D+(p-qp)^D=((p-q)^D)^3(-q)+((p-q)^D)^3p=((p-q)^D)^2.  \eqno(3.1)$$

We replace $p$, $q$ by $1-p$ and $q$ in the equality (3.1) to get
$$(pq)^D+((1-q)(1-p))^D=((1-p-q)^D)^2. \eqno(3.2) $$

Multiplying the equality (3.2) by $p$ on the left yields
$$p(pq)^D+p((1-q)(1-p))^D=p((1-p-q)^D)^2. \eqno(3.3) $$

Note that $p(pq)^D=p(pq)(pq)^D(pq)^D=(pq)^D$ and Theorem 3.6. We have $$(pq)^D=(pqp)^D-p((1-q)(1-p))^D.$$

(3) By Lemma 2.3, we have
\begin{center}
$(pqp)^Dpq=pq((pq)^D)^2pq=(pq)^Dpq$.
\end{center}

Completing the proof.
\end{proof}

\begin{theorem} Let $1-pq\in R^D$. Then $p-q\in R^D$ and
\begin{center}
$(p-q)^D=(1-pq)^D(p-pq)+(p+q-pq)^D(pq-q)$.
\end{center}
\end{theorem}
\begin{proof}
By Theorem 3.5(5), we have
$$(1-pq)^D=1-p+[(p-q)^D]^2[p+pq(1-p)]+[\sum_{i=0}^{k-1}(p-q)^\pi(p-q)^{2i}]pq(p-1).\eqno(3.4)$$

we replace $p$ and $q$ by $1-p$ and $1-q$ respectively in the equality (3.4) to obtain
$$(p+q-pq)^D=p+[(p-q)^D]^2[1-p+(1-p)(1-q)p]+[\sum_{i=0}^{k-1}(p-q)^\pi(p-q)^{2i}](1-p)(1-q)p.\eqno(3.5)$$

Multiplying the equality (3.4) by $p-pq$ on the right yields $$(1-pq)^D(p-pq)=p(p-q)^D=(p-q)^D(1-q).\eqno(3.6)$$

Multiplying the equality (3.5) by $pq-p$ on the right yields
$$(p+q-pq)^D(pq-q)=(p-q)^Dq. \eqno(3.7)$$

Notice that (3.6) and (3.7). One has
\begin{eqnarray*}
(1-pq)^D(p-pq)+(p+q-pq)^D(pq-q) &=& (p-q)^D(1-q)+(p-q)^Dq\\
&=&(p-q)^D.
\end{eqnarray*}

The proof is complete.
\end{proof}

An involution $x\mapsto x^*$ in a ring $R$ is an anti-isomorphism of degree 2, that is,
 \begin{center}
 $(a^*)^*=a$, $(a+b)^*=a^*+b^*$, $(ab)^*=b^*a^*$
\end{center}
for all $a, b \in R$. Idempotent element $p$ is a \emph{projector} if $p=p^*$. An element $a\in R$ is \emph{$\ast$-cancellable} if
\begin{center}
$a^*ax=0\Rightarrow ax=0$ and $xaa^*=0\Rightarrow xa=0$.
\end{center}
A ring $R$ is \emph{$\ast$-reducing} if all elements in $R$ are $\ast$-cancellable. This is equivalent to $a^*a=0\Rightarrow a=0$ for all $a\in R$.

\begin{theorem} Let $R$ be a $\ast$-reducing ring and $p$, $q$ be two projectors. Then\\
$(1)$ $(p-q)^D=p-q$ if and only if $pq=qp$,\\
$(2)$ If $6\in R^{-1}$, then $(p+q)^D=p+q$ if and only if $pq=qp=0$.
\end{theorem}

\begin{proof} (1) If $pq=qp$, it is easy to check that $(p-q)^D=p-q$.

Conversely, $(p-q)^D=p-q$ implies that $(p-q)^3=p-q$, that is $pqp=qpq$. Since $R$ is a $\ast$-reducing ring and $(pq-qp)^*(pq-qp)=0$, it follows that $pq=qp$.

(2) If $pq=qp=0$, by [6, Corollary 1], $(p+q)^D=p+q$.

Conversely, $(p+q)^D=p+q$ implies $(p+q)^3=p+q$, i.e.,
$$2pq+2qp+pqp+qpq=0. \eqno(3.8)$$

Multiplying the equality (3.8) by $p$ on the left yields
$$2pq+3pqp+pqpq = 0.\eqno(3.9)$$

Multiplying the equality $(3.8)$ by $q$ on the right yields
$$2pq+3qpq+pqpq=0. \eqno(3.10)$$

Combining the equalities $(3.9)$ and $(3.10)$, we obtain that $pqp=qpq$. By the proof of (1), we get $pq=qp$. Hence, equality $(3.8)$ can be simplified to $6pq=0$. Therefore, $pq=qp=0$.
\end{proof}

Let $p$, $q$ be two idempotents in a Banach algebra. Then, $p+q$ is Drazin invertible if and only if $p-q$ is Drazin invertible [9]. However, in general, this need not be true in a ring. For example, let $R=\mathbb{Z}$ and $p=q=1$. Then $p-q=0$ is Drazin invertible, but $p+q=2$ is not Drazin invertible. Next, we consider what conditions $p$ and $q$ satisfy, $p-q\in R^D$ implies that $p+q\in R^D$. Deng and Wei [5] proved the following results for bounded linear operators in Banach spaces. Now we present it in a ring without proofs.

\begin{theorem} Let $p-q\in R^D$. If $F$, $G$ and $H$ are given by Definition $3.1$ and $(p+q)(p-q)^\pi \in R^{\rm nil}$, then\\
$(1)$ $(p+q)^D=(p-q)^D (p+q)(p-q)^D,$\\
$(2)$ $(p-q)^D=(p+q)^D (p-q)(p+q)^D,$\\
$(3)$  $(p-q)^\pi=(p+q)^\pi,$\\
$(4)$ $(p-q)^D= F+G-H,$\\
$(5)$ $(p+q)^D=(2G-H)(F+G-H).$
\end{theorem}

\begin{theorem} Let $p-qp\in R^D$. Then \\
$(1)$ $(p-q)^D=(p-q)^2((p-qp)^D-(q-qp)^D),$\\
$(2)$ If $(p+q)(p-q)^\pi\in R^D$, we have $p((p+q)^D-(p-q)^D)(p-q)^2=0$.
\end{theorem}

\begin{proof} Since $(p-qp)^D=((p-q)^D)^3p$ and $(q-qp)^D=q((q-p)^D)^3$, we get
\begin{eqnarray*}
 (p-q)^2((p-qp)^D-(q-qp)^D) &=& (p-q)^2(((p-q)^D)^3p-q((q-p)^D)^3) \\
   &=& (p-q)^Dp+q(p-q)^D\\
   &=& (p-q)^Dp+(p-q)^D(1-p)\\
   &=& (p-q)^D.
  \end{eqnarray*}
(2) Note that $(p+q)^D=(p-q)^D(p+q)(p-q)^D$. We have
\begin{eqnarray*}
 p((p+q)^D-(p-q)^D)(p-q)^2 &=& p(p-q)^D(p+q)(p-q)-p(p-q) \\
   &=& p(1-q+1-p)(p-q)^D(p-q)-p(p-q)^2(p-q)^D\\
   &=& p(p-q)^2(p-q)^D-p(p-q)^2(p-q)^D\\
   &=& 0.
   \end{eqnarray*}
\end{proof}

In Theorem 3.10, we know that if $p-q\in R^D$ and $(p+q)(p-q)^\pi$ is nilpotent, then $p+q\in R^D$. Naturally, does the reverse statement above hold? Next, we give an example to  illustrate it is not true.
Take $R= \mathbb{Z}_7$, \begin{center}
 $p=\left(
                                                                                             \begin{array}{cc}
                                                                                               1 & 0 \\
                                                                                               0 & 1 \\
                                                                                             \end{array}
                                                                                           \right)
 \in M_2(R)$ and $q=\left(
 \begin{array}{cc}
            1 & 0 \\
            0 & 0  \\
          \end{array}
          \right)
\in M_2(R)$,
 \end{center}
 $p$ and $q$ are obvious two idempotents. Moreover, \begin{center}
 $p+q=\left(
 \begin{array}{cc}
                                                   2 & 0 \\
                                                   0 & 1 \\
                                                 \end{array}
                                                 \right)
 $, $p-q=\left(
           \begin{array}{cc}
             0 & 0 \\
             0 & 1 \\
           \end{array}
         \right)
 $.
 \end{center}
Hence, $p+q$ and $p-q$ are Drazin invertible. However, $(p+q)(p-q)^\pi$ is not nilpotent.

Koliha and Rako\v{c}evi\'{c} [10] proved that $p-q\in R^{-1}$ implies that $p+q\in R^{-1}$ for idempotents $p$ and $q$ in a ring $R$. Moreover, if $2\in R^{-1}$, then $p-q\in R^{-1}$ is equivalent to $p+q \in R^{-1}$ and $1-pq\in R^{-1}$. Hence, we have the following result.
\begin{corollary} $[12, {\rm Theorem~2.2}]$ Let $p-q\in R^{-1}$. If $F=p(p-q)^{-1}$ and $G=(p-q)^{-1}p$, then\\
$(1)$ $(p+q)^{-1}=(p-q)^{-1} (p+q)(p-q)^{-1},$\\
$(2)$ $(p-q)^{-1}=(p+q)^{-1}(p-q)(p+q)^{-1},$\\
$(3)$ $(p-q)^{-1}= F+G-1,$\\
$(4)$ $(p+q)^{-1}=(2G-1)(F+G-1)$.
\end{corollary}

\centerline {\bf ACKNOWLEDGMENTS} This research is supported by the National Natural Science Foundation of China (10971024),
the Specialized Research Fund for the Doctoral Program of Higher Education (20120092110020),
and the Natural Science Foundation of Jiangsu Province (BK2010393) and the Foundation of Graduate
Innovation Program of Jiangsu Province(CXLX13-072).

\end{document}